\documentclass[reqno]{amsart}

\usepackage[utf8]{inputenc}
\usepackage[T1]{fontenc}

\usepackage{amsmath,amssymb,amsfonts,amsthm}
\usepackage{mathtools}
\usepackage{graphicx}
\usepackage{enumerate}
\usepackage{lscape}
\usepackage{dsfont}
\usepackage{color}
\usepackage{setspace}
\usepackage{hyperref}

\newcommand{\R}{\mathds{R}}

\newcommand{\C}{\mathds{C}}
\newcommand{\CP}{\mathds{C}\mathrm{P}}
\newcommand{\CH}{\mathds{C}\mathrm{H}}

\newcommand{\K}{K\"{a}hler}

\newcommand{\tr}{\operatorname{tr}}

\newcommand{\Kill}{{\mathrm{Kill}}}
\newcommand{\ad}{\operatorname{ad}}
\newcommand{\Ad}{\operatorname{Ad}}
\newcommand{\diag}{\operatorname{diag}}
\newcommand{\sgn}{{\operatorname{sgn\,}}}

\newtheorem{theor}{Theorem}[section]

\newtheorem{defin}{Definition}
\newtheorem{lem}[theor]{Lemma}
\newtheorem{cor}{Corollary}

\newtheorem{remark}{Remark}

\begin{document}
\title[Universal embeddings of flag manifolds and rigidity phenomena]{Universal embeddings of flag manifolds and rigidity phenomena}

\author{Andrea Loi}
\address{(Andrea Loi) Dipartimento di Matematica \\ Universit\`a di Cagliari (Italy)}
\email{loi@unica.it}

\author{Roberto Mossa}
\address{(Roberto Mossa) Dipartimento di Matematica \\ Universit\`a di Cagliari (Italy)}
\email{roberto.mossa@unica.it}

\author{Fabio Zuddas}
\address{(Fabio Zuddas) Dipartimento di Matematica \\ Universit\`a di Cagliari (Italy)}
\email{fabio.zuddas@unica.it}

\thanks{The authors are supported by INdAM and GNSAGA - Gruppo Nazionale per le Strutture Algebriche, Geometriche e le loro Applicazioni; by GOACT - Funded by Fondazione di Sardegna; and partially funded by PNRR e.INS Ecosystem of Innovation for Next Generation Sardinia (CUP F53C22000430001, codice MUR ECS00000038).}



\begin{abstract}
We prove a universal embedding theorem for flag manifolds: \emph{every flag manifold admits a holomorphic isometric embedding into an irreducible classical flag manifold}. This result generalizes the classical celebrated embedding theorems of Takeuchi \cite{Take78} and Nakagawa--Takagi \cite{NaTa76}.
Using this embedding, we establish new rigidity phenomena for holomorphic isometries between homogeneous \K\ manifolds.  As a first immediate  consequence we show the triviality of a \K-Ricci soliton submanifod of  $\mathcal C \times \Omega$, where $\mathcal C$ is a flag manifold and $\Omega$ is a homogeneous bounded domain.

Secondly, we show that no \emph{weak-relative} relationship can occur among the fundamental classes of homogeneous \K\ manifolds: flat spaces, flag manifolds, and homogeneous bounded domains. Two \K\ manifolds are said to be \emph{weak relatives} if they share, up to local isometry, a common \K\ submanifold of complex dimension at least two. Our main result precisely shows that if $\mathcal{E}$ is (possibly indefinite) flat, $\mathcal{C}$ is a flag manifold, and $\Omega$ is a homogeneous bounded domain, then:
\begin{enumerate}
  \item $\mathcal{E}$ is not weak relative to $\mathcal{C}\times\Omega$;
  \item $\mathcal{C}$ is not weak relative to $\mathcal{E}\times\Omega$;
  \item $\Omega$ is not weak relative to $\mathcal{E}\times\mathcal{C}$.
\end{enumerate}
This extends, in two independent directions, the rigidity theorem of Loi-Mossa \cite{LMrighom}: we pass from \emph{relatives} to the more flexible notion of \emph{weak relatives} and dispense with the earlier ``special'' restriction on the flag-manifold factor. This result also unifies previous rigidity results from the literature, e.g., \cite{Cal,Cheng2021,CDY,diloi,HY,HYbook,UmearaC}.\end{abstract}

\maketitle
\tableofcontents  

\section{Introduction}\label{intro}
The study of holomorphic isometric embeddings between homogeneous \K\ manifolds occupies a central place in complex differential geometry, as it has interactions between curvature, symmetry, and complex structure. Recall that a \K\  manifold $S$ is called \emph{homogeneous} if its group of biholomorphic isometries acts transitively on $S$. Fundamental classes are homogeneous bounded domains, flat homogeneous \K\ manifolds, and generalized flag manifolds (hereafter \emph{flag manifolds}); the latter are compact, simply connected, homogeneous \K\ manifolds. A \emph{homogeneous bounded domain} is a bounded domain $\Omega\subset\C^n$ endowed with a homogeneous \K\ metric $g_\Omega$ (for instance the Bergman metric when $\Omega$ is a bounded symmetric domain). Note that a given bounded domain may admit many distinct homogeneous \K\  metrics not equivalent to the Bergman metric; see \cite{DISCALALOIISHI} for details. Holomorphic isometric embeddings between such spaces have led to celebrated results, beginning with the foundational work of Calabi~\cite{Cal} (see also~\cite{LoiZedda-book}), and continuing with the classical embedding theorems of Takeuchi~\cite{Take78} and Nakagawa--Takagi~\cite{NaTa76}. The problem of embedding a given homogeneous \K\ manifold into a more symmetric or higher-dimensional ambient space plays a pivotal role in the classification and rigidity theory of \K\ manifolds with large symmetry groups.
We also recall the  foundational contributions of Mok to the theory of holomorphic isometries in the noncompact setting;
we refer the reader to the following (non-exhaustive) references \cite{Mok1989, Mok2012, Mok2016}.
The reader is further referred to the recent work of Ishi~\cite{Ishi24}, where it is shown that any simply-connected homogeneous domain in $\C^n$ admits a holomorphic isometric embedding into a Siegel--Jacobi domain. For the infinite-dimensional setting, in which the target is a complex Hilbert space form, we refer to~\cite{DISCALALOIISHI}.

In this paper, we prove a universal holomorphic isometric embedding theorem for flag manifolds (see Theorem \ref{mainteoremb} below), which generalizes and extends the classical embedding results  of Takeuchi and Nakagawa--Takagi mentioned above.
In those works the ambient space is the complex projective space $\C P^N$
equipped with the Fubini-Study metric; consequently, an embedding is possible
only if the \K\ form of the source flag manifold is integral.
Moreover, \cite{NaTa76} assumes that the source is a
Hermitian symmetric space of compact type.
Our main theorem removes all these requirements: it applies to
\emph{every} source  flag manifold, and allows a
classical irreducible flag manifold of any type as  ambient space.

\begin{theor}\label{mainteoremb}
Let $\mathcal C$ be a flag manifold.  
Fix one of the classical groups
\[
   \tilde G\;\in\;\bigl\{\mathrm{SU}(N),\ \mathrm{SO}(N),\ \mathrm{Sp}(N)\bigr\}.
\]
Then, for some \emph{sufficiently large} integer \(N\),  
there exist an irreducible classical flag manifold  
\(\tilde{\mathcal C}=\tilde G/\tilde K\) 
and a holomorphic isometric embedding  
$\Phi :\; \mathcal C \longrightarrow \tilde{\mathcal C}.$
\end{theor}

This result not only refines previous  embedding theorems but it also has important implications 
for the rigidity of \K\ submanifolds (for an overview of several rigidity questions and further references, see \cite{LMrighom}). 

The first immediate consequence is the following corollary for induced \K-Ricci solitons (KRS).

\begin{cor}\label{corolmainteor}
Let $(g,X)$ be a KRS on a \K\ submanifold of  $\mathcal C \times \Omega$, where $\mathcal C$ is a flag manifold and $\Omega$ is a homogeneous bounded domain.  
Then the soliton is \emph{trivial}, i.e.\ $g$ is \K-Einstein.  
In particular, any KRS induced by a flag manifold $\mathcal C$ is trivial.
\end{cor}

Corollary~\ref{corolmainteor} generalizes both
\cite[Theorem~1.1]{LMrighom} and
\cite[Corollary~1]{LMrighom}, where the ambient factor was required to be ``special''.  
The strategy is to embed the product
$
\mathcal C \times \Omega \;\hookrightarrow\; \tilde{\mathcal C} \times \Omega,
$
holomorphically and isometrically, where $\tilde{\mathcal C}$ is a classical flag manifold.
The existence of such an embedding is guaranteed by Theorem~\ref{mainteoremb}; thus, the cited results apply verbatim and yield the triviality of the soliton.

\begin{remark}\label{rem:cannot-extend-flat}\rm
The conclusion of Corollary~\ref{corolmainteor} fails if one replaces the factor $\Omega$ (or $\mathcal C$) by a flat \K\ manifold. Indeed, as observed in \cite[Remark~1]{LMrighom}, there exist nontrivial homogeneous KRS on $\C\times\CP^1$ and on $\C\times\CH^1$. 
\end{remark}
For completeness, although these facts are not used in our proofs, we note that a KRS induced by a flat (possibly indefinite) space is \K-Einstein and hence flat, by combining \cite[Theorem~1]{LMpams} with Umehara's theorem~\cite{UmearaEinstein}, while a KRS induced by a homogeneous bounded domain is \K--Einstein by \cite[Theorem~1.1(i)]{PRIMO} (cf. also Lemma \ref{lemmaricciflat} below).

Our second main rigidity contribution is  Theorem \ref{mainteorrel} below, where  we examine rigidity questions for pairs of K\"ahler manifolds that are relatives, or weak relatives, in the sense of Di~Scala--Loi \cite{diloi}. We recall the definitions.
 
\begin{defin}[relatives and weak relatives]\label{defrel}
Let $S_1$ and $S_2$ be K\"ahler manifolds.
\begin{enumerate}[(1)]
  \item $S_1$ and $S_2$ are \emph{relatives} if there exists a K\"ahler manifold $M$
        and holomorphic isometries $\varphi_i \colon M \to S_i$ for $i=1,2$.
  \item $S_1$ and $S_2$ are \emph{weak relatives} if there exist
        two locally isometric (not necessarily biholomorphic) K\"ahler manifolds $M_1$ and $M_2$ 
        and holomorphic isometries $\varphi_i \colon M_i \to S_i$ for $i=1,2$.
\end{enumerate}
\end{defin}

Clearly, any pair of relative \K\ manifolds is also a pair of weak relatives.
Moreover, if $M_1$ and $M_2$ are one-dimensional K\"ahler manifolds and $f \colon M_1 \to M_2$ is an isometry, then $f$ is automatically holomorphic or anti\-holomorphic (see  \cite[Lemma 6]{Placini2023}). Thus, in complex dimension~$1$ the notions of weak relatives and relatives coincide.
In general, however, weak relatives need not be relatives. Examples can be constructed in complex dimension~$2$; see \cite{Placini2023}. It is presently unknown whether there exist weak relatives K\"ahler manifolds of complex dimension $>2$ that are not relatives.  For up-to-date results on relative K\"ahler manifolds we refer to the survey \cite{YUANYUAN}; see also \cite{ChengHaoYuanZhang2024,LMblowup}. For the case of weak relatives, consult \cite{diloi,Placini2023}.

\begin{theor}\label{mainteorrel}
Let $\mathcal E$ be a (possibly indefinite) flat \K\ manifold, let $\mathcal C$ be a flag manifold, and let $\Omega$ be a homogeneous bounded domain. Then:
\begin{itemize}
\item[(i)] $\mathcal E$ is not weak relative to the \K\ product $\mathcal C \times \Omega$;
  \item[(ii)] $\mathcal C$ is not weak relative to the \K\ product $\mathcal E \times \Omega$;
  \item[(iii)] $\Omega$ is not weak relative to the \K\ product $\mathcal E \times \mathcal C$.
\end{itemize}
\end{theor}

This theorem strengthens Loi--Mossa's rigidity result\,\cite[Theorem 1.1]{LMrighom} in two mutually independent directions, thereby refining  earlier contributions by Calabi, Cheng, Chen--Deng--Yuan, Di~Scala--Loi, Huang--Yuan, and Umehara\,\cite{Cal,Cheng2021,CDY,diloi,HY,HYbook,UmearaC}.

First, the notion of \emph{relatives} is replaced by  \emph{weak relatives}, so the earlier setting appears as a particular instance of the present one.  Second, the ``special'' condition formerly imposed on the target is dispensed with: the target needs no longer be either a classical flag manifold or a flag manifold  with a  K\"ahler metric which is  a positive  multiple of a projectively induced homogeneous metric, that is, the pull--back of the Fubini-Study metric via a holomorphic isometric immersion into projective space.  Although that hypothesis already covered homogeneous K\"ahler--Einstein metrics
and  every flag manifold with $b_2=1$ (and hence every compact Hermitian symmetric space  is ``special'' in that sense), it excluded numerous cases such as flag manifolds of exceptional type and their K\"ahler products.  
Theorem \ref{mainteorrel}, by contrast, applies to \emph{all} flag manifolds.

\begin{remark}\rm
There can exist three \K\ manifolds $M_i$ ($i=1,2,3$) such that $M_1$ is not relative to either $M_2$ or $M_3$ individually, yet $M_1$ is relative to the product $M_2\times M_3$; see \cite{LMrighom} for explicit examples.
Consequently, the proof of Theorem~\ref{mainteorrel} cannot be reduced to examining each factor separately. 
\end{remark}

As an immediate and striking consequence of Theorem~\ref{mainteorrel} we obtain a strong extension of \cite[(ii) of Theorem~1.1]{PRIMO} and \cite[Corollary~2]{LMrighom}:
\begin{cor}\label{maincorol2}
Any two among a flat space $\mathcal E$, a flag manifold $\mathcal C$, and a homogeneous bounded domain $\Omega$ are not weak relatives.
\end{cor}

\vskip 0.3cm 
To orient the reader, we briefly outline how the paper is organized and highlight the main ideas behind each theorem.
Section \ref{prel} collects Lie-theoretic preliminaries, including the Killing form and an adapted root-space decomposition for flag manifolds. Section \ref{proofs} contains the proofs of our main results.
Central for the proof of Theorem \ref{mainteoremb}  is the choice of Cartan subalgebras that are \emph{aligned} by an injective Lie--group homomorphism $G \hookrightarrow \tilde G,$
where \(G\) denotes the compact, connected, \emph{simple} Lie group acting transitively on the source flag manifold \(\mathcal C = G/K\), while \(\tilde G\) is one of the classical simple groups \(\mathrm{SU}(N),\ \mathrm{SO}(N)\  (N\neq 4),\ \mathrm{Sp}(N)\) acting on the ambient flag manifold \(\tilde{\mathcal C} = \tilde G/\tilde K\).
Once this alignment is fixed, Lemma~\ref{Killingfund} supplies the precise rescaling of the Killing form that makes the pulled--back and ambient K\"ahler forms coincide, and a root--by--root check then confirms that the resulting map is holomorphic. 

The proof of Theorem~\ref{mainteorrel} relies on two key ingredients.  
First, Theorem~\ref{mainteoremb} allows us to replace a general flag manifold 
\(\mathcal{C}\) with its holomorphic and isometric image inside a classical flag manifold 
\(\widetilde{\mathcal{C}}\). The rigidity theorem of Loi--Mossa~\cite[Thm.~1.1]{LMrighom} 
then ensures that none of the pairs
\[
(\mathcal{E},\ \widetilde{\mathcal{C}} \times \Omega),\qquad
(\widetilde{\mathcal{C}},\ \mathcal{E} \times \Omega),\qquad
(\Omega,\ \mathcal{E} \times \widetilde{\mathcal{C}})
\]
can share a positive-dimensional Kähler submanifold, hence, they are not relatives.
Secondly, assuming that one of the original pairs were only weak relatives, 
the method developed in~\cite{Placini2023}, together with Lemma~\ref{lemmaricciflat}, which 
states that any Ricci-flat \K\ submanifold of \(\mathcal{E} \times \Omega\) is flat, leads 
to the conclusion that the pair must in fact be relatives, yielding a contradiction and 
thus completing the proof.


\section{Preliminaries}\label{prel}

\subsection{The Killing form}
Let $\mathfrak g$ be a finite-dimensional Lie algebra over a field
$F$ of characteristic \(0\).
The \emph{Killing form} of $\mathfrak g$ is the bilinear map
\[
  \Kill_{\mathfrak g}\colon
  \mathfrak g\times\mathfrak g\longrightarrow F,
  \qquad
  \Kill_{\mathfrak g}(X,Y)\;:=\;
  \operatorname{tr}\!\bigl(\operatorname{ad}_X\circ\operatorname{ad}_Y\bigr),
\]
where
\(
  \operatorname{ad}_X(Y):=[X,Y]
\)
is the adjoint representation.
It is routine to verify that $\Kill_{\mathfrak g}$ is symmetric,
bilinear, and associative; i.e.  
\[
  \Kill_{\mathfrak g}\bigl([X,Y],Z\bigr)
  \;=\;
  \Kill_{\mathfrak g}\bigl(X,[Y,Z]\bigr),
  \qquad\forall\,X,Y,Z\in\mathfrak g.
\]

Recall the celebrated Cartan's criterium asserting that the Killing form of $\mathfrak g$ is non-degenerate
if and only if $\mathfrak g$ is \emph{semisimple}.

In this paper we will be concerned primarily with the case where $\mathfrak g$
is \emph{simple}, i.e.\ non-abelian and possessing no ideals other than
$\,\{0\}$ and $\mathfrak g$ itself.
Notice that a semisimple Lie algebra decomposes as a finite direct sum of simple ideals.
Recall also  that  a real simple Lie algebra $\mathfrak{g}$ is said to be  {\em compact} 
if  its  Killing form $\Kill_{\mathfrak g}$ is negative-definite (equivalently, $\mathfrak g$ is the Lie algebra
 of a compact connected Lie group). 
 We shall use the following lemma in the proof of Theorem \ref{mainteoremb}.
 
\begin{lem}\label{Killingfund}
Let $ \mathfrak{g}$ be a \emph{compact simple} Lie algebra over $\R$ and let  
\[
   b: \mathfrak{g}\times \mathfrak{g}\longrightarrow\R
\]
be an \emph{associative} symmetric bilinear form, i.e.  
\begin{equation}\label{eq:associative}
  b([X,Y],Z)=b(X,[Y,Z])
  \qquad\forall\,X,Y,Z\in \mathfrak{g}.
\end{equation}
Then there exists a constant $c\in\R$ such that
\begin{equation}\label{eq:bKill}
  b(X,Y)\;=\;c\,\Kill_{\mathfrak{g}}(X,Y),
  \qquad\forall\,X,Y\in \mathfrak{g}.
\end{equation}
\end{lem}
\begin{proof}
Because $ \mathfrak{g}$ is compact and simple, its complexification
$ \mathfrak{g}^\C$ is again simple, as it can easily be seen by using the uniqueness up to conjugacy of compact real forms.
Extend $b$ complex-bilinearly to
$b_\C: \mathfrak{g}^\C\times \mathfrak{g}^\C\longrightarrow\C$; 
equation \eqref{eq:associative} implies that $b_\C$ is still associative.
Since $\mathfrak{g}^\C$ is simple, Schur's lemma for the (irreducible) adjoint
representation tells us that every associative symmetric bilinear form is a scalar
multiple of the Killing form; hence
\begin{equation}\label{eq:bCKill}
   b_\C(X,Y)=c\,\Kill_{\mathfrak{g}^\C}(X,Y)
   \qquad\text{for some }c\in\C .
\end{equation}

For $X,Y\in\mathfrak{g}$ we have $\Kill_{\mathfrak{g}^\C}(X,Y)=\Kill_\mathfrak{g}(X,Y)$
(\cite[Lemma~6.1, p.~180]{Helgason}); thus, restricting
\eqref{eq:bCKill} gives
\[
   b(X,Y)=c\,\Kill_{\mathfrak{g}}(X,Y), \qquad X,Y\in\mathfrak{g} .
\]
Both $b$ and $\Kill_{\mathfrak{g}}$ are real-valued on $\mathfrak{g}$, forcing $c\in\R$.
\end{proof}

\subsection{Irreducible flag manifold}
Let \(\mathcal C\) be an \emph{irreducible} flag manifold and let \(\omega\) denote the \K\  form of a homogeneous \K\ metric \(g\) on \(M\), i.e. $g(X,Y)\;=\;\omega\!\bigl(X,\,JY\bigr)$, where $J$
is the integrable almost complex structure.

It is classical (see, for example, \cite{Arva1996flag,BFR1986}) that
$\mathcal C \;=\; G/K,$
where \(G\) is a compact \emph{simple} Lie group, i.e. its Lie algebra is a compact simple Lie algebra,  and
\begin{equation}\label{defK}
   K \;=\; C_G(Z_0) \;=\; \bigl\{\gamma\in G \mid \Ad(\gamma)(Z_0)=Z_0\bigr\}
\end{equation}
is the centraliser of a non-zero element \(Z_0\in\mathfrak{g}\).
The element \(Z_0\) completely determines both the complex and the \K\ structures on \(\mathcal C\) as follows.

Choose the reductive decomposition, orthogonal with respect to the Killing form~\(\Kill_{\mathfrak{g}}\),
\begin{equation}\label{reductive}
   \mathfrak{g} \;=\; \mathfrak{k} \;\oplus\; \mathfrak{m},
\end{equation}
where $\mathfrak{k}$ is the Lie algebra of $K$.
Identifying \(T_o(G/K)\) with \(\mathfrak{m}\) at the base point \(o=[e]_K\), one has
\begin{equation}\label{omega0}
   \omega_o(X,Y)
   \;=\;
   \Kill_{\mathfrak{g}}\!\bigl([Z_0,X],Y\bigr),
   \qquad X,Y\in\mathfrak{m}.
\end{equation}

Conversely, for any compact simple Lie group \(G\) and any \(Z_0\in\mathfrak{g}\), the space \(G/K\) defined by~\eqref{defK}, endowed with the \(2\)-form~\eqref{omega0}, is an irreducible flag manifold carrying precisely the homogeneous \K\ structure determined by \(Z_0\).

Let \(\mathfrak{h}\subset\mathfrak{g}^{\C}\) be a Cartan subalgebra, i.e.\ a maximal abelian subalgebra consisting of elements that act diagonalisably in the adjoint representation.  Since the operators \(\ad_H\;(H\in\mathfrak{h})\) are simultaneously diagonalisable, there exist common eigenvectors \(X\in\mathfrak{g}^{\C}\) such that
\[
   [H,X]=\alpha(H)\,X.
\]
The non-zero linear functionals \(\alpha:\mathfrak{h}\to\C\) form the finite root system of \(\mathfrak{g}^{\C}\); each \emph{root space}
\[
   \mathfrak{g}^{\alpha}
   \;=\;
   \bigl\{X\in\mathfrak{g}^{\C}\mid[H,X]=\alpha(H)X\;\ \forall\,H\in\mathfrak{h}\bigr\}
\]
is one-dimensional and spanned by \emph{root vectors}.

Henceforth, we denote by \(E_\alpha\) a fixed generator of \(\mathfrak{g}^\alpha\); it will be written out explicitly only when necessary.

Then we can write
\begin{equation}\label{ghgalpha}
\mathfrak{g}^{\C}=\mathfrak{h}\oplus \sum_{\alpha} \mathfrak{g}^{\alpha}
\end{equation}

Every Cartan subalgebra arises by complexifying a maximal torus of \(\mathfrak{g}\).
For our purpose it will be crucial to consider Cartan subalgebras arising from maximal tori of $\mathfrak{k}$
and containing  the centre \(\mathfrak{t}=C(\mathfrak{k})\) of \(\mathfrak{k}\). This is possible thanks to the following (well-known) lemma included for reader convenience:

\begin{lem}\label{lemmaxtori}
Let \(\mathfrak t_{\max}\subset \mathfrak k\) be a maximal torus of the compact subalgebra \(\mathfrak k\). Then
\begin{enumerate}[(i)]
  \item  \(\mathfrak t:=C(\mathfrak k)\) is contained in \(\mathfrak t_{\max}\);
  \item \(\mathfrak t_{\max}\) is a maximal torus of the Lie algebra \(\mathfrak g\).
\end{enumerate}
\end{lem}

\begin{proof}
\textbf{(i)}\; Fix a maximal torus \(\mathfrak t_{\max}'\subset\mathfrak k\)
such that  $\mathfrak t\subseteq \mathfrak t_{\max}'$. For every other maximal torus \(\mathfrak t_{\max}\subset\mathfrak k\) there exists \(\gamma\in K\) such that 
\[
  \mathfrak t_{\max}=\operatorname{Ad}(\gamma)\,\mathfrak t_{\max}'.
\]
Because \(\mathfrak t\) is stable under the adjoint action of \(K\), we obtain
\[
  \mathfrak t=\operatorname{Ad}(\gamma)\,\mathfrak t\subseteq\operatorname{Ad}(\gamma)\,\mathfrak t_{\max}'=\mathfrak t_{\max},
\]
which proves part~(i).

\medskip
\noindent
\textbf{(ii)}\; Let \(\mathfrak s_{\max}\subset\mathfrak g\) be a maximal torus containing \(\mathfrak t_{\max}\). Choose any element \(Z\in\mathfrak s_{\max}\). From part~(i) we know \(\mathfrak t\subseteq\mathfrak t_{\max}\); in particular, \(Z_0\in\mathfrak t\subseteq\mathfrak s_{\max}\). Hence
\[
  [Z_0,Z]=0 \quad\Longrightarrow\quad Z\in C_{\mathfrak g}(Z_0)=\mathfrak k.
\]
Since \(\mathfrak t_{\max}\) is maximal in \(\mathfrak k\), we must have \(Z\in\mathfrak t_{\max}\). Consequently \(\mathfrak s_{\max}=\mathfrak t_{\max}\), and \(\mathfrak t_{\max}\) is a maximal torus of \(\mathfrak g\).
\end{proof}

To describe the complex structure of a flag manifold $(\mathcal C=G/K, g)$ we proceed as follows.
Let $Z_0$ be as in \eqref{defK}.
Declare a root \(\alpha\) to be \emph{black} if it does not vanish identically on \(\mathfrak{t}\) (easily seen to be equivalent to \(\alpha(Z_0)\neq0\)), and \emph{white} otherwise (\(\alpha(Z_0)=0\)).
With the complexified decomposition
\begin{equation}\label{gC}
   \mathfrak{g}^{\C}= \mathfrak{k}^{\C}\oplus\mathfrak{m}^{\C},
\end{equation}
white root vectors lie in \(\mathfrak{k}^{\C}\) while black root vectors span $ \mathfrak{m}^{\C}$, i.e.
\begin{equation}\label{mcblack}
   \mathfrak{m}^{\C} \;=\; 
\sum_{\substack{\alpha \ \text{black}}} \mathfrak{g}^{\alpha}.
\end{equation}
It follows by \eqref{ghgalpha} that


\begin{equation}\label{tcblack}
   \mathfrak{g}^{\C} \;=\; 
\mathfrak{h}\oplus\sum_{\substack{\alpha \ \text{white}}} \mathfrak{g}^{\alpha}\oplus 
\sum_{\substack{\alpha \ \text{black}}} \mathfrak{g}^{\alpha}.
\end{equation}

Finally, for each black root $\alpha$ the almost complex structure \(J\) acts by
\begin{equation}\label{defJ}
J\!\bigl(E_\alpha\bigr)
  \;=\;
  i\,\operatorname{sgn}\!\bigl[\alpha(i Z_0)\bigr]\,E_\alpha,
\end{equation}
where \(E_{\alpha}\in\mathfrak{g}^{\alpha}\) is a root vector.
Notice that  
$\alpha(i Z_0)$ is indeed  a real number, see e.g.  \cite[p. 254]{Knapp2013lie}.

\section{Proofs of the main results}\label{proofs}
We begin by proving  Theorem~\ref{mainteoremb}; the argument is organised into the following steps.

\begin{proof}[Proof of Theorem \ref{mainteoremb}]

\mbox {}
\paragraph{\bf Step 1. (the case when $\mathcal C$ is irreducible)}
Assume $\mathcal C$ is irreducible, i.e. \(\mathcal C=G/K\) with \(G\) compact, connected and \emph{simple}.  
Thus,  for $N$ sufficiently large
there exists a injective homomorphism
\begin{equation}\label{homomorphism}
\phi :G\rightarrow \tilde G,
\end{equation}
where $\tilde G$ is a classical  irreducible simple group $SU(N)$, $Sp(N)$ or $SO(N)\ (N\neq 4)$. 
Indeed, for every compact simple Lie group $G$ one can find an injective homomorphism into $SU(N)$ for $N$ large  
and $\label{iota}
\iota: SU(N) \rightarrow Sp(2N) \cap SO(2N)$
defined by 
$$\iota(A + i B) = \left( \begin{array}{cc}
A & B \\
-B & A
\end{array} \right)$$
is an injective  Lie group homomorphism.

Define
\[
\tilde K:=\bigl\{\tilde\gamma\in\tilde G \,\big|\, 
          \operatorname{Ad}_{\tilde\gamma}\!\bigl(d\phi(Z_0)\bigr)=d\phi(Z_0)\bigr\}.
\]
Then \(\tilde {\mathcal C}:=\tilde G/\tilde K\) is an irreducible flag manifold of classical type and
 it easily  seen that    $\phi(K) \subseteq \tilde K$.
 
Then
\begin{equation}\label{PHImap}
\Phi: \mathcal C \rightarrow \mathcal {\tilde C}, \ \ \Phi[\gamma]_K = [\phi(\gamma)]_{\tilde K}
\end{equation}
is a well-defined injective smooth equivariant map, 
i.e. $\Phi \circ \lambda = \phi(\lambda) \circ \Phi$, for every $\lambda \in G$. 

\smallskip

\paragraph{\bf Step 2 (aligned Cartan subalgebras)}

Take $d\phi$ as the differential  at the identity $e\in G$ of  the homomorphism $\phi$ given by  
\eqref{homomorphism}.
The  Lie algebra homomorphism $d \phi: \mathfrak{g} \rightarrow\tilde{\mathfrak{g}}$, where $\mathfrak{g}$  and 
$\tilde{\mathfrak{g}}$ are the Lie algebras of $G$ and $\tilde G$ respectively, extends 
to a Lie algebra homomorphism  $d \phi: \mathfrak{g}^{\C} \rightarrow \mathfrak{\tilde g}^{\C}$ (which with a little abuse of notation we denote by the same notation).

\medskip 
We claim that  there  exist Cartan subalgebras
$\mathfrak h \subset \mathfrak g,\
  \tilde{\mathfrak h} \subset \tilde{\mathfrak g},$
such that
\[
  \mathfrak t := C(\mathfrak k) \subset \mathfrak h, \qquad
  \tilde{\mathfrak t} := C(\tilde{\mathfrak k}) \subset \tilde{\mathfrak h}, \qquad
  d\phi(\mathfrak h) \subset \tilde{\mathfrak h}.
\]

To prove the claim let $\mathfrak t_{\max} \subset \mathfrak k$ be a maximal torus. By Lemma~\ref{lemmaxtori} it contains the centre $\mathfrak t = C(\mathfrak k)$ and is maximal in $\mathfrak g$. Consequently, its complexification 
$\mathfrak h := \mathfrak t_{\max}^{\mathbb C}$
is a Cartan subalgebra of $\mathfrak g$ with $\mathfrak t \subset \mathfrak h$.

Because $d\phi(\mathfrak k) \subset \tilde{\mathfrak k}$ and $\mathfrak t_{\max} \subset \mathfrak k$, the image $d\phi(\mathfrak t_{\max})$ is an abelian subalgebra of $\tilde{\mathfrak k}$, that is, a torus. Choose a maximal torus $\tilde{\mathfrak t}_{\max} \subset \tilde{\mathfrak k}$ that contains $d\phi(\mathfrak t_{\max})$. Again by Lemma~\ref{lemmaxtori}, we have $\tilde{\mathfrak t} \subset \tilde{\mathfrak t}_{\max}$ and $\tilde{\mathfrak t}_{\max}$ is maximal in $\tilde{\mathfrak g}$. Hence its complexification
$\tilde{\mathfrak h} := \tilde{\mathfrak t}_{\max}^{\mathbb C}$
is a Cartan subalgebra of $\tilde{\mathfrak g}$ with $\tilde{\mathfrak t} \subset \tilde{\mathfrak h}$.

Finally,
\[
  d\phi(\mathfrak h)
    = d\phi\bigl(\mathfrak t_{\max}^{\mathbb C}\bigr)
    = d\phi(\mathfrak t_{\max})^{\mathbb C}
    \subset \tilde{\mathfrak t}_{\max}^{\mathbb C}
    = \tilde{\mathfrak h},
\]
so all the desired inclusions hold, and the claim is proved.

\begin{remark}\rm\label{rmknew}
Since any two maximal tori in $\tilde{\mathfrak{g}}$ are conjugate, by composing $\phi$ with the conjugation by a suitable $\gamma \in \tilde G$ we can always assume that the Cartan subalgebra $\tilde{\mathfrak{h}}$ is the standard Cartan subalgebra of the classical group $\tilde G$ consisting of the diagonal matrices in $\tilde{\mathfrak{g}}^{\C}$ (see, e.g.  \cite[Sect. 2. Ch. III]{Helgason}).
\end{remark}

\smallskip 

\paragraph{\bf Step 3 (correspondence of root spaces)}
Let  $\mathfrak{h}$ (resp. $\tilde{\mathfrak{h}}$) be a Cartan subalgebra of $\mathfrak{g}^{\C}$ (resp. $\tilde{\mathfrak{g}}^{\C}$) such that  $d \phi(\mathfrak{h})\subseteq \tilde{\mathfrak{h}}$ as in Step 2.

\begin{defin}[{\(\phi\)-related roots}]
Let \(\alpha\) be a root of \(\mathfrak{g}^{\C}\).
A root \(\tilde\alpha\) of \(\tilde{\mathfrak{g}}^{\C}\) is said to be
\emph{\(\phi\)-related to \(\alpha\)} if
\[
\tilde\alpha\!\bigl(d\phi(H)\bigr)=\alpha(H)
\quad\text{for every }H\in\mathfrak{h}.
\]
\end{defin}

\begin{lem}
 Fix a root \(\alpha\) of \(\mathfrak{g}^{\C}\) with respect to \(\mathfrak{h}\) and a root
vector \(E_\alpha\).
Then, there exist roots \(\tilde\alpha_1,\dots,\tilde\alpha_k\) of \(\tilde{\mathfrak{g}}^{\C}\),  $\phi$-related to $\alpha$, and non-zero complex numbers
\(c_1,\dots,c_k\) such that
\begin{equation}\label{linearcomb}
d\phi(E_\alpha)=\sum_{j=1}^{k} c_j E_{\tilde\alpha_j},
\end{equation}
Consequently, 
\(\alpha\) is black (resp.\ white) for \(G/K\) iff each \(\tilde\alpha_j\) is black (resp.\ white) for
\(\tilde G/\tilde K\) and 
\begin{equation}\label{inclm}
d\phi(\mathfrak{m}) \subset \tilde{\mathfrak{m}}.
\end{equation}
where
\(\mathfrak{g}=\mathfrak{k}\oplus\mathfrak{m}\) and
\(\tilde{\mathfrak{g}}=\tilde{\mathfrak{k}}\oplus\tilde{\mathfrak{m}}\)
are the Killing-orthogonal splittings.
\end{lem}

\begin{proof}
Because \([H,E_\alpha]=\alpha(H)E_\alpha\) for all \(H\in\mathfrak{h}\), applying \(d\phi\) gives
\begin{equation}\label{phiHalpha}
[d\phi(H),d\phi(E_\alpha)] = \alpha(H)\,d\phi(E_\alpha).
\end{equation}

 Decompose
\begin{equation}\label{deocmposdfiE}
d\phi(E_\alpha)=\tilde H_0 + \sum_{j=1}^{k} c_j E_{\tilde\alpha_j},
\end{equation}
with \(\tilde H_0\in\tilde{\mathfrak{h}}\) and
\(E_{\tilde\alpha_j}\in\tilde{\mathfrak{g}}^{\tilde\alpha_j}\) linearly independent.
We claim that not all $c_j$ vanish.
Indeed if $d\phi (E_\alpha)=\tilde H_0$ then,  by using   $d\phi(H)\in \tilde{\mathfrak{h}}$,  $[d \phi(H), \tilde H_0] = 0$ (since $\tilde{\mathfrak{h}}$ is abelian) one gets
$$[d\phi (E_\alpha), d\phi (H)]=[\tilde H_0, d\phi (H)]=0$$
for all $H\in {\mathfrak{h}}$, which implies 
$$d\phi([E_\alpha, H])=\alpha (H)d\phi (E_\alpha)=0,$$
namely $\alpha (H)=0$ for all $H$ in contrast with the definition of roots.

Applying \(\operatorname{ad}_{d\phi(H)}\) to \eqref{deocmposdfiE} gives
\[
\alpha(H)\,d\phi(E_\alpha)=
\sum_{j=1}^{k} c_j\,\tilde\alpha_j\bigl(d\phi(H)\bigr) E_{\tilde\alpha_j}.
\]
Subtracting \eqref{deocmposdfiE} multiplied by \(\alpha(H)\) yields
\[
\alpha (H)\tilde H_0+\sum_{j=1}^{k} c_j
  \bigl[ \alpha(H)-\tilde\alpha_j\bigl(d\phi(H)\bigr) \bigr] E_{\tilde\alpha_j}=0
  \quad(\forall H\in\mathfrak{h}). 
\]
Independence of the \(E_{\tilde\alpha_j}\) forces $\tilde H_0$  to vanish and 
\begin{equation}\label{eigenvalues}
\tilde\alpha_j\!\bigl(d\phi(H)\bigr)=\alpha(H)\quad\text{for every }H\in\mathfrak{h},
\end{equation}
proving that $\alpha_j$'s are $\phi$-related to $\alpha$.
Equation \eqref{deocmposdfiE}  (with \(\tilde H_0=0\)) gives the linear combination \eqref{linearcomb}.
Setting \(H=Z_0\) in \eqref{eigenvalues} implies
that  the colours of  $\alpha_j$ and $\tilde\alpha_j$ are  preserved.
Finally, one gets
\[
\mathrm{d}\phi\bigl(\mathfrak{m}^{\mathbb{C}}\bigr)
   \;=\;
   \sum_{\alpha\,\text{ black}}
   \mathrm{d}\phi\bigl(\mathfrak{g}^{\alpha}\bigr)
   \;\subseteq\;
   \sum_{\substack{\tilde{\alpha}\\\text{$\phi$-related to }\alpha}}
   \tilde{\mathfrak{g}}^{\tilde{\alpha}}
   \;\subseteq\;
   \sum_{\tilde{\alpha}\,\text{ black}}
   \tilde{\mathfrak{g}}^{\tilde{\alpha}}
   \;=\;
   \tilde{\mathfrak{m}}^{\mathbb{C}}.
\]
and \eqref{inclm} follows.
\end{proof}

\smallskip

\paragraph{\bf Step 4 (calibrating the \K\ forms)}
Define the pulled-back Killing form
\[
\phi^{*}\Kill_{\tilde{\mathfrak{g}}}(X,Y):=
  \Kill_{\tilde{\mathfrak{g}}}\!\bigl(d\phi(X),d\phi(Y)\bigr),
  \qquad X,Y\in\mathfrak{g}.
\]

Now, $\phi^*(\Kill_{\tilde{\mathfrak{g}}})$ is easily seen to be an associative bilinear form on $\mathfrak{g}$.
Then,  since  $\mathfrak{g}$ is a compact simple algebra, we can conclude by Lemma \ref{Killingfund} and by  the fact that the Killing form are non degenerate and both negative definite that
there exist a   {\em positive} real constant $c$ such that 

\begin{equation}\label{constantc}
\phi^*(\Kill_{\tilde{\mathfrak{g}}})(X, Y) = c\  \Kill_{\mathfrak{g}}(X, Y)
\end{equation}

Let \(\omega\) be the invariant \K\ form on \(\mathcal C\) determined by \(Z_0\); i.e.
\[
\omega_o(X,Y)=\Kill_{\mathfrak{g}}([Z_0,X],Y),\qquad X,Y\in\mathfrak{m}.
\]
Similarly \(d\phi(Z_0)\) determines an invariant form
\(\hat\omega\) on \(\tilde {\mathcal C}\) with
\[
\hat\omega_{\tilde o}(\tilde X,\tilde Y)=
  \Kill_{\tilde{\mathfrak{g}}}\!\bigl([d\phi(Z_0),\tilde X],\tilde Y\bigr),
  \qquad \tilde X,\tilde Y\in\tilde{\mathfrak{m}}.
\]
Set \(\tilde\omega:=c^{-1}\hat\omega\).
Then  that by  \eqref{constantc} and \eqref{inclm},  one easily gets:

\begin{equation}\label{lambdaZU(N)3}
\Phi^*(\tilde \omega)_o(X, Y) = \omega_o(X, Y)
\end{equation}
for every $X, Y \in \mathfrak{m}$.
Due the fact that $\Phi:\mathcal C\to \tilde {\mathcal C}$ is equivariant and 
the \K\ metrics are invariant we get that $\Phi$ is symplectic,
i.e. $\Phi ^*\tilde\omega=\omega$.

\smallskip

\paragraph{\bf Step 5 (holomorphicity of \(\Phi\))}
Fix a black root \(\alpha\) and its root vector \(E_\alpha\).
By Step 3
\[
d\phi(E_\alpha)=\sum_{j=1}^{k} c_j E_{\tilde\alpha_j},
\quad
\operatorname{sgn}\!\bigl(\tilde\alpha_j(i d\phi(Z_0))\bigr)=
\operatorname{sgn}\!\bigl(i\,\alpha(Z_0)\bigr).
\]
By \eqref{defJ} the complex structures act by
\[
J(E_\alpha)=
  i\,\operatorname{sgn}\!\bigl(\alpha(iZ_0)\bigr)E_\alpha,\qquad
\tilde J(E_{\tilde\alpha_j})=
  i\,\operatorname{sgn}\!\bigl(\tilde\alpha_j(i d\phi(Z_0))\bigr)E_{\tilde\alpha_j}.
\]
Therefore
\[
d\phi\bigl(JE_\alpha\bigr)=
  i\,\operatorname{sgn}\!\bigl(\alpha(i Z_0)\bigr)
  \sum_{j} c_j E_{\tilde\alpha_j}
  =\tilde J\Bigl(\sum_{j} c_j E_{\tilde\alpha_j}\Bigr)
  =\tilde J\bigl(d\phi(E_\alpha)\bigr).
\]
Linearity gives $d\phi \circ J=\tilde J\circ d\phi$ on
\(\mathfrak{m}^{\C}\);  $G$-equivariance extends it over \(\mathcal C\),  i,e, 
$d\Phi\circ J=\tilde J\circ d\Phi$, so \(\Phi\) is holomorphic.

\smallskip 

\paragraph{\bf Step 6 (isometry)}
A holomorphic symplectomorphism between \K\  manifolds preserves the metric; hence
\[
\Phi^{*}\tilde g = g,
\]
and the proof  when $\mathcal C$ is irreducible is  complete.

\smallskip

\paragraph{\bf Step 7 ($\mathcal C$  reducible)}
If the flag manifold \((\mathcal C=G/K,g)\) is \emph{reducible}, namely 
\[
 (\mathcal C,g)\;=\;(G_1/K_1,g_1)\times\cdots\times(G_r/K_r,g_r),
\]
with each \((G_i/K_i,g_i)\), \(i=1,\dots,r\), irreducible, then by the previous steps we already have a holomorphic and isometric embedding
\[
\mathcal C
\;\longrightarrow\;
 SU(N_1)/\tilde K_1 \times \cdots \times SU(N_r)/\tilde K_r .
\]
Consequently, to obtain an \emph{irreducible} ambient space of the form
\(\tilde G/\tilde K\) with
\(\tilde G\in\{\mathrm{SU}(N),\mathrm{SO}(N),\mathrm{Sp}(N)\}\),
it suffices to construct an embedding
\[
SU(N_1)/\tilde K_1 \times \cdots \times SU(N_r)/\tilde K_r
\;\hookrightarrow\;
SU(N)/\tilde K,
\]
for suitable integers \(N\) and a subgroup 
\(\tilde K\); once this is achieved, Step 1 applies verbatim and yields the desired result.
This can be achieved by arguments analogous to those used in the irreducible case, now fully explicit. We merely sketch the proof.

Let us consider the natural block embedding
$$\psi: SU(N_1) \times \cdots \times SU(N_r) \rightarrow SU(N), \ \ (A_1, \dots, A_r) \mapsto 
\begin{pmatrix}
A_1 & 0   & \cdots &        & 0      \\
0   & A_2 &        & \cdots & 0      \\
     &     & \ddots &        &        \\
0   & 0   &        & \cdots & 0      \\
0   & 0   &        & 0      & A_r
\end{pmatrix}$$

where $N = N_1 + \cdots + N_r$.

By Remark \ref{rmknew}
we can take as Cartan subalgebra of 
$$(\mathfrak{su}_{N_1} \oplus  \cdots \oplus \mathfrak{su}_{N_r} )^{\C} = \mathfrak{sl}_{N_1}(\C) \oplus  \cdots \oplus \mathfrak{sl}_{N_r}(\C)$$ the standard subalgebra given by   the set $\mathfrak{h}_1 \oplus \cdots \oplus \mathfrak{h}_r$ of the block matrices $H = \diag(H^1, \dots, H^r)$ where $H^i$ is diagonal of order $N_i$ with $\tr(H^i)=0$, and we clearly have $d \psi(\mathfrak{h}_1 \oplus \cdots \oplus \mathfrak{h}_r) \subseteq \mathfrak{h}$ where $\mathfrak{h}$ is the standard Cartan subalgebra of $\mathfrak{su}_{N}^{\C} = \mathfrak{sl}_{N}(\C)$ consisting of the complex diagonal matrices with null trace (in the terminology introduced above, $\mathfrak{h}_1 \oplus \cdots \oplus \mathfrak{h}_r$ and $\mathfrak{h}$ are aligned).
On the one hand, the roots $\tilde \alpha$ of $\mathfrak{sl}_{N}(\C)$ with respect to $\mathfrak{h}$ are the functionals $\tilde \alpha = \tilde \alpha_{ij}$ defined by $\tilde \alpha_{ij}(H) = H_{ii} - H_{jj}$ for some $i, j = 1, \dots, N$, $i \neq j$, and the root vector $E_{\tilde \alpha}$ corresponding to $\tilde \alpha = \tilde \alpha_{ij}$ can be choosen as  the matrix which has $1$ at the entry $ij$ and $0$ otherwise.
On the other hand, the roots $\alpha$ of $\mathfrak{sl}_{N_1}(\C) \oplus  \cdots \oplus \mathfrak{sl}_{N_r}(\C)$ with respect to $\mathfrak{h}_1 \oplus \cdots \oplus \mathfrak{h}_r$  are the functionals $\alpha = \alpha^k_{ij}$ defined for every $H = \diag(H^1, \dots, H^r)$ by $\alpha^k_{ij}(H) = H^k_{ii} - H^k_{jj}$ for some $k = 1, \dots, r$ and $i,j = 1, \dots, N_k$, $i \neq j$, i.e.
\begin{equation}\label{H1Hk}
\alpha^k_{ij}(H^1, \dots, H^r) = \alpha_{ij}(H^k)
\end{equation}
With a slight abuse of notation, for the sake of simplicity we will write this equality also as $\alpha(H^1, \dots, H^r) = \alpha(H^k)$ whenever $\alpha = \alpha^k_{ij}$.
The root vector corresponding to $\alpha = \alpha^k_{ij}$ can be chosen as the diagonal block matrix 
$$E_{\alpha} = \diag(0, \dots, E_{ij}^k, \dots, 0),$$
where $E_{ij}^k$ is the matrix in the block $\mathfrak{su}_{N_k}^{\C}$ having $1$ as entry $ij$ and $0$ otherwise.

From this explicit description of the roots it is clear that for every root $\alpha$ of $\mathfrak{sl}_{N_1}(\C) \oplus  \cdots \oplus \mathfrak{sl}_{N_r}(\C)$ there exists exactly one $\psi$-related root $\tilde \alpha$ of $\mathfrak{sl}_{N}(\C)$, i.e. such that $\tilde \alpha(d \psi(H)) = \alpha(H)$ for every $H \in \mathfrak{h}_1 \oplus \cdots \oplus \mathfrak{h}_r$, and that $d \psi(E_{\alpha}) = E_{\tilde \alpha}$ for the unique root $\tilde \alpha$ which is $\psi$-related to $\alpha$.

Now, assume that $\tilde K_j = C_{SU(N_j)}(Z_j)$ for some $Z_ j \in \mathfrak{su}(N_j)$, and that each factor $SU(N_j)/\tilde K_j$ is endowed with the Kahler structure determined by $Z_j$ as in the previous steps. This means that $SU(N_1)/\tilde K_1 \times \cdots \times SU(N_r)/\tilde K_r$ is endowed with the product Kahler form $\omega = \omega^1 + \cdots + \omega^r$ such that its restriction to $o = ([e]_{K_1}, \dots, [e]_{K_r})$ is

\begin{equation}\label{prodKform}
\omega_o(X, Y) = \sum_{j=1}^r \omega_o^j(X_j, Y_j) = \sum_{j=1}^r \Kill_{SU(N_j)} ([Z_j, X_j], Y_j)
\end{equation}

for every 
$$X = \sum_{j=1}^r X_j, \ Y = \sum_{j=1}^r Y_j  \in  \mathfrak{m}_1 \oplus \cdots \oplus \mathfrak{m}_r \simeq T_o(SU(N_1)/\tilde K_1 \times \cdots \times SU(N_r)/\tilde K_r),$$ where $\mathfrak{m}_i$ denotes as above the orthogonal complement (with respect to the Killing form) of $\tilde{\mathfrak{k}}_i$ in $\mathfrak{su}(N_i)$.
Moreover, the complex structure $J$ on $SU(N_1)/\tilde K_1 \times \cdots \times SU(N_r)/\tilde K_r$ is defined for every black root $\alpha = \alpha^k_{ij}$, $s = 1, \dots, N_j$, by 
\begin{equation}\label{H1Hk2}
J(E_{\alpha}) = i \cdot \sgn(i \alpha(Z_1, \dots, Z_r)) E_{\alpha} = i \cdot \sgn(i \alpha(Z_k)) E_{\alpha}
\end{equation}
where we have used (\ref{H1Hk}).

Now, let

$$\tilde Z := d \psi(c_1 Z_1, \dots, c_r Z_r), \ \ c_j = \frac{N}{N_j}$$

and let $\tilde K = C_{SU(N)}(\tilde Z)$. Let us endow $SU(N)/\tilde K$ with the Kahler form $\tilde \omega$ and the complex structure $\tilde J$ defined by $\tilde Z$.
Notice that if $\alpha = \alpha^k_{ij}$ is a black (resp. white) root, then its $\psi$-related root $\tilde \alpha$ is also black (resp. white). Indeed, by definition $\alpha$ is black (resp. white) if and only if $\alpha(Z_1, \dots, Z_r) = \alpha(Z_k) \neq 0$ (resp. $\alpha(Z_k) = 0$). Then, again by (\ref{H1Hk}) we have
\begin{equation}\label{H1Hk3}
\tilde \alpha(\tilde Z) = \tilde \alpha(d \psi(c_1 Z_1, \dots, c_r Z_r)) = \alpha(c_1 Z_1, \dots, c_r Z_r) = \alpha(c_k Z_k) = c_k \alpha(Z_k)
\end{equation}
and then the claim follows. As in the irreducible case, this implies that $d \phi(\mathfrak{m}_1 \oplus \cdots \oplus \mathfrak{m}_r ) \subseteq \mathfrak{m}$.
Finally , it is easy to see that the map
$$\Psi: SU(N_1)/\tilde K_1 \times \cdots \times SU(N_r)/\tilde K_r \rightarrow SU(N)/\tilde K$$
$$\Psi([A_1]_{K_1}, \dots, [A_r]_{K_r}) = [\psi(A_1, \dots, A_r)]_{\tilde K}$$
 is well-defined and injective, and one can prove by similar arguments as in the irreducible case that $\Psi$ preserves both the symplectic and the complex structures.
This concludes Step 7 and thereby completes the proof of the theorem.
 \end{proof}

\medskip

We are now in a position to prove Theorem~\ref{mainteorrel}.  
We need a  lemma.

\begin{lem}\label{lemmaricciflat}
Let \(F\) be a Ricci-flat \K\ manifold admitting a holomorphic  isometric  immersion
$F \longrightarrow \mathcal E \times \Omega,$
where \(\mathcal E\) is a flat complex Euclidean space and \(\Omega\) is a bounded homogeneous domain.  
Then \(F\) is flat.
\end{lem}

\begin{proof}
Due  to the work of Ishi \cite[Th. 1]{Ishi24} we know that every bounded homogeneous domain \(\Omega\) admits a holomorphic and isometric immersion into a Siegel upper half-space \(\mathfrak S\) of type~III.  
Classical results of Kor\'anyi-Wolf \cite{KW65} show that \(\mathfrak S\) is \K\ equivalent, via a Cayley transformation, to a bounded Cartan domain of the same type, whose Ricci tensor is non-positive.  
Then, by the Gauss equation it follows that  $F\rightarrow \mathcal E \times \Omega\rightarrow 
 \mathcal E\times \mathfrak S$ is totally geodesic. 
Since a totally geodesic submanifold of a locally homogeneous Riemannian manifold is also locally homogeneous, a well- known theorem of Alekseevsky-Kimel’fel’d-Spiro (see \cite{AlekseevskyKimelfeld75} and \cite{Spiro93}) implies 
$F$ is flat.
\end{proof}


\begin{proof}[Proof of Theorem~\ref{mainteorrel}]
We first show that none of the pairs 
\begin{equation}\label{pairs}
(\mathcal E,\ \mathcal C\times\Omega),\qquad
(\Omega,\ \mathcal E\times\mathcal C),\qquad
(\mathcal E\times\Omega,\  \mathcal C)
\end{equation}
are relatives. Subsequently, we prove that if any of them are weak relatives, then they must in fact be relatives.
By Theorem~\ref{mainteoremb}, the flag manifold $\mathcal C$ admits a holomorphic and isometric embedding into a classical flag manifold $\widetilde{\mathcal C}$. Therefore, it suffices to examine the corresponding pairs
\[
(\mathcal E,\ \widetilde{\mathcal C}\times\Omega),\qquad
(\Omega,\ \mathcal E\times \widetilde{\mathcal C}),\qquad
(\mathcal E\times\Omega,\  \widetilde{\mathcal C}).
\]
In each case, the rigidity theorem of Loi-Mossa~\cite[Thm~1.2]{LMrighom} (valid for classical flag manifolds) rules out the existence of a positive-dimensional common Kähler submanifold; hence, no relatives exist.

Assume now that one of the pairs in \eqref{pairs} consists of weak relative \K\ manifolds. By
the very definition of weak relatives (see (2) of Definition \ref{defrel} in the Introduction), there exist two locally isometric \K\ manifolds $M_1$ and $M_2$, together with holomorphic isometries $\varphi_i \colon M_i \to S_i$, for \(i = 1,2\), where \((S_1, S_2)\) is one of the pairs in \eqref{pairs}. Let $\varphi \colon M_1 \to M_2$ be a local isometry between the two. Since we are only interested in the local behavior, we may assume without loss of generality that $\varphi$ is a global isometry.
We aim to modify $\varphi$ so as to obtain a holomorphic isometry $\tilde{\varphi} \colon M_1 \to M_2$, which would imply that \(S_1\) and \(S_2\) are relatives, contradicting the first part of the proof.

Let $M_1 = F \times N_1 \times \cdots \times N_k$
be the de Rham decomposition of $M_1$, where $F$ collects all Ricci-flat factors. We claim that the factor $F$ cannot appear; that is, $\dim F = 0$.
Indeed, since \(S_1\) is either \(\mathcal{E}\), \(\Omega\), or \(\mathcal{E}\times\Omega\), and \(F\) admits a \K\ immersion into \(M_1\) and hence into \(S_1\), Lemma~\ref{lemmaricciflat} implies that \(F\) is flat.
If \(S_1 = \Omega\), then \(F\) cannot appear, since \(\Omega\) is not relative to \(\mathcal{E}\), by \cite{PRIMO}.  
If instead \(S_1 = \mathcal{E}\) or \(\mathcal{E} \times \Omega\), then \(S_2\) is respectively 
\(\mathcal{C} \times \Omega\) or \(\mathcal{C}\). But then \(\varphi(F)\) is a flat Kähler submanifold of \(M_2\) and hence of \(S_2\), which again contradicts the first part, since \(\mathcal{E}\) is not relative to either \(\mathcal{C} \times \Omega\) or \(\mathcal{C}\). This proves the claim.
Now, using \cite[Lemma~6]{Placini2023}, which states that any isometry between irreducible, non-Ricci-flat Kähler manifolds is either holomorphic or anti-holomorphic, we can proceed as in the proof of \cite[Thm~3]{Placini2023} and modify $\varphi$ on each irreducible factor $N_j$ to obtain the desired holomorphic isometry $\tilde\varphi \colon M_1 \to M_2$.
The theorem is thus proved.
\end{proof}

By the proof of Theorem \ref{mainteorrel} and by Lemma \ref{lemmaricciflat}
we obtain the following corollary.

\begin{cor}
Let $S$ be a \K\ manifold that is weak relative to a bounded homogeneous domain $\Omega$. Then $S$ and $\Omega$ are relatives. 
\end{cor}

A similar result holds true if one replaces $\Omega$ by a 
 projective \K\ manifold (see \cite[Th.~3]{Placini2023}). 
Extending this result to flag manifolds is a natural and intriguing direction, which we plan to explore in future work.


\begin{thebibliography}{99}

\bibitem{AlekseevskyKimelfeld75}
D.~V.~Alekseevsky e B.~N.~Kimel'fel'd, 
\emph{Structure of homogeneous Riemannian spaces with zero Ricci curvature},  
\textit{Funktsional. Anal. i. Prilozhen.} 9 (2) (1975), 5-11.  


\bibitem{ArezzoLiLoi_GHlimits}
C.~Arezzo, C.~Li, A.~Loi,
\emph{Gromov--Hausdorff limits and Holomorphic isometries},
\textit{Mathematics Research Letters} (to appear), arXiv: 2306.16113, 19 pp. (v2, 9 January 2024).

\bibitem{Arva1996flag}
A.~Arvanitoyeorgos,
``Geometry of flag manifolds'',
\emph{International Journal of Geometric Methods in Modern Physics}
\textbf{3}\,(5-6) (2006), 957-974.\ %
\href{https://doi.org/10.1142/S0219887806001399}{doi:10.1142/S0219887806001399}

\bibitem{BFR1986}
M.~Bordemann, M.~Forger\ and H.~R\"{o}mer,
``Homogeneous K\"{a}hler manifolds: paving the way towards new supersymmetric sigma models'',
\emph{Communications in Mathematical Physics}
\textbf{102} (1986), 605--647.\ %
\href{https://doi.org/10.1007/BF01221650}{doi:10.1007/BF01221650}




\bibitem{Cal} E. Calabi,  {\em Isometric Imbedding of Complex Manifolds}, Ann. of Math. 58 (1953), no.~1.

\bibitem{Cheng2021}
X. Cheng, Y. Hao, \emph{On the non-existence of common submanifolds of \K\ manifolds and complex space forms}, Ann. Global Anal. Geom. 60 (2021), no.~1, 167--180.


\bibitem{CDY}
X. Cheng, A. J. Di Scala, Y. Yuan,  
\emph{\K\ submanifolds and the Umehara algebra}, 
Int. J. Math. 28 (2017), no.~4, 1750027, 13 pp.

\bibitem{ChengHaoYuanZhang2024}
X.~Cheng, Y.~Hao, Y.~Yuan, and X.~Zhang,
Complex submanifolds of indefinite complex space forms,
\emph{Proc. Amer. Math. Soc.} 152 (2024), no.~6, 2541--2550.
\newblock DOI: \href{https://doi.org/10.1090/proc/16743}{10.1090/proc/16743}.


\bibitem{diloi} 
A. J. Di Scala, A. Loi, 
{\em \K\ manifolds and their relatives},
Ann. Sc. Norm. Super. Pisa Cl. Sci. (5) 9 (2010), no.~3, 495--501.

\bibitem{DISCALALOIISHI}
A. J. Di Scala, H. Ishi, A. Loi, \emph{\K\ immersions of homogeneous \K\ manifolds into complex space forms}, Asian J. Math. 16 (2012), 479--487.

\bibitem{Helgason}
S.~Helgason,
\textit{Differential Geometry, Lie Groups, and Symmetric Spaces}.
Vol.~80 of \textit{Pure and Applied Mathematics},
Academic Press, New York London, 1978 \textit{Graduate Studies in Mathematics},
vol.~34, American Mathematical Society, Providence (RI), 2001, xvi + 641 pp.


\bibitem{HY}
 X. Huang, Y. Yuan, \emph{Holomorphic isometry from a \K\ manifold into a product of complex projective manifolds}, Geom. Funct. Anal. 24 (2014), no.~3, 854--886. 
 
\bibitem{HYbook}
 X. Huang, Y. Yuan,
 \emph{Submanifolds of Hermitian symmetric spaces}, Analysis and geometry, 197--206,
Springer Proc. Math. Stat. 127, Springer, Cham, 2015.

\bibitem{Ishi1}
H.~Ishi, 
\emph{Holomorphic isometric embeddings of the complex unit ball into bounded symmetric domains}, 
Proc. Japan Acad. Ser. A Math. Sci. \textbf{85} (2009), no.~10, 123-127.

\bibitem{Ishi2}
H.~Ishi, 
\emph{Proper holomorphic isometric embeddings of the unit ball into symmetric domains}, 
Ann. Mat. Pura Appl. (4) \textbf{192} (2013), no.~1, 77-92.


\bibitem{Ishi24}
H.~Ishi,
\textit{On a concrete realization of simply connected complex domains admitting homogeneous \K\ metrics},
in \textit{The Bergman Kernel and Related Topics} (HSSCV 2022),
K.~Hirachi et al. (eds.),
Springer Proc. Math. Stat. 447, Springer, Singapore, 2024, 261--272.



\bibitem{Knapp2013lie}
A.~W. Knapp,
\newblock \emph{Lie Groups Beyond an Introduction},
\newblock Progress in Mathematics, Birkh\"auser Boston, 2013.
\newblock ISBN: 9781475724530.

\bibitem{KW65} 
Korányi, A.  and  Wolf, J.A. 
\emph{Realization of Hermitian symmetric spaces as generalized half-planes}. \textit{Annals of Mathematics. Second Series}, \textbf{81} (1965), 265--288.

\bibitem{LMblowup}
A.~Loi, R.~Mossa,
{\em On holomorphic isometries into blow-ups of $\mathbb{C}^n$},
 \emph{Mediterr. J. Math.} 20 (2023), no.~4, Paper No.~230, 11~pp.

\bibitem{LMpams} 
A. Loi, R. Mossa, \emph{\K\ immersions of \K-Ricci solitons into definite or indefinite complex space forms}, Proc. Amer. Math. Soc. 149 (2021), no.~11, 4931--4941.






\bibitem{PRIMO}
A.\ Loi, R.\ Mossa,
\emph{Holomorphic isometries into homogeneous bounded domains},
Proc.\ Amer.\ Math.\ Soc.\ 151 (2023), no.~9, 3975--3984.

\bibitem{LMrighom}
A.\ Loi, R.\ Mossa,
\emph{Rigidity properties of holomorphic isometries into homogeneous \K\ manifolds},
Proc.\ Amer.\ Math.\ Soc.\ 152 (2024), no.~7, 3051--3062.

\bibitem{LoiZedda-book} 
A. Loi, M. Zedda, \emph{K\"{a}hler Immersions of K\"{a}hler Manifolds into Complex Space Forms}, Lecture Notes of the Unione Matematica Italiana 23, Springer, 2018.


\bibitem{Mok1989}
N. Mok, \textit{Metric Rigidity Theorems on Hermitian Locally Symmetric Manifolds}, Series in Pure Mathematics, Vol.~6, World Scientific, 1989.

\bibitem{Mok2012}
N. Mok, \textit{Extension of germs of holomorphic isometries up to normalizing constants with respect to the Bergman metric,} \textit{Journal of the European Mathematical Society}, \textbf{14}(5), 2012, pp.~1617--1656.

\bibitem{Mok2016}
N. Mok, ``Holomorphic isometric embeddings of bounded symmetric domains into irreducible bounded symmetric domains,'' \textit{European Journal of Mathematics}, \textbf{2}, 2016, pp.~827--868.


\bibitem{NaTa76}
H.~Nakagawa and R.~Takagi, 
\textit{On locally symmetric \K\ submanifolds in a complex projective space}, 
\emph{Journal of the Mathematical Society of Japan} 28 (1976), no.~4, 638--667.

\bibitem{Placini2023}
G.~Placini,
{\em On weakly and strict relatives K\"ahler manifolds},
preprint (2023), arXiv:2306.16174 [math.DG].
\newblock Available at \href{https://arxiv.org/abs/2306.16174}{arXiv:2306.16174}.

\bibitem{Spiro93}
A.~Spiro, 
\emph{A remark on locally homogeneous Riemannian spaces},  
Results Math. \textbf{24}\,(3–4) (1993), 318–325. 

\bibitem{Take78}
M.~Takeuchi, 
\textit{Homogeneous \K\ submanifolds in complex projective spaces}, 
\emph{Japanese Journal of Mathematics} \textbf{4} (1978), no.~1, 171--219.

\bibitem{UmearaEinstein}
M.~Umehara,
\textit{Einstein K\"ahler submanifolds of a complex linear or hyperbolic space},
\emph{Tohoku Mathematical Journal (2)} \textbf{39} (1987), no.~3, 385--389.


\bibitem{UmearaC} 
M. Umehara, 
{\em \K\ submanifolds of complex space forms},
Tokyo J. Math. 10 (1987), 203--214.

\bibitem{YUANYUAN} 
Y. Yuan,
{\em  Local holomorphic isometries, old and new results},
Proc. 7th Int. Congr. Chinese Math., Vol. II, 409--422, Int. Press, Somerville, MA, 2019.




\end{thebibliography}
\end{document}